\documentclass{amsart}

\usepackage{amssymb, amsmath}
\usepackage{mathrsfs}
\usepackage{amscd}
\usepackage{verbatim}

\theoremstyle{plain}
\newtheorem{theorem}{Theorem}
\theoremstyle{remark}
\newtheorem{remark}[theorem]{Remark}

\theoremstyle{plain}

\newtheorem{lemma}[theorem]{Lemma}
\newtheorem{proposition}[theorem]{Proposition}

\def\R{{\mathbb R}}

\newcommand{\E}{{\mathbb E}}
\renewcommand{\P}{{\mathbb P}}

\renewcommand{\O}{\Omega}

\newcommand{\one}{{{\bf 1}}}

\newcommand{\lb}{\langle}
\newcommand{\rb}{\rangle}

\begin{document}

\title
[Khintchine inequalities with weights]{On Khintchine inequalities with a
weight}

\begin{abstract}
In this note we prove a weighted version of the Khintchine inequalities.
\end{abstract}

\author{Mark Veraar}
\address{Delft Institute of Applied Mathematics\\
Delft University of Technology \\ P.O. Box 5031\\ 2600 GA Delft\\The
Netherlands} \email{mark@profsonline.nl, M.C.Veraar@tudelft.nl}




\maketitle

Let $(\O, \mathcal{A}, \P)$ be a probability space and let $(r_n)_{n\geq 1}$ be
a Rademacher sequence. For a random variable $\xi:\O\to \R$ and $p>0$ we write $\|\xi\|_p = (\E|\xi|^p)^{1/p}$.
Our main result is the following weighted version of Khintchine's
inequality. We also allow the weight to be zero on a set of positive measure.

\begin{theorem}\label{thm:main}
Let $0<p<\infty$ and let $w\in L^{q}(\O)$ for some $q>p$, and assume $s:=\P(w\neq 0)>2/3$. Let $\xi = \sum_{n\geq 1} r_n x_n$ with $\sum_{n\geq 1} x_n^2<\infty$.
Then there exist
constants $C_1:=C_1(p,w), C_2:= C_2(p,w)>0$ such that
\begin{equation}\label{eq:khintchineweigth}
 C^{-1}_1
\Big(\sum_{n\geq 1} x_n^2\Big)^{\frac12} \leq \|w\xi\|_p \leq C_2
\Big(\sum_{n\geq 1} x_n^2\Big)^{\frac12}.
\end{equation}
Consequently, the $p$-th moments for $0<p<q$ are all comparable.
\end{theorem}

If $w\equiv 1$ the result reduces the Khintchine inequalities \cite{Khin}.
Although the weighted version of the result is easy to prove, to our knowledge it was not known, and potentially useful for others.
We need a well-known $L^0$-version of Khintchine's inequality. We provide the details to obtain explicit constants.

\begin{proposition}\label{prop:KhintchineL0}
For all $a\in (0,1)$ and for all $(x_n)_{n\geq
1}$ in $\ell^2$, one has
\[\P\Big(\Big|\sum_{n\geq 1} r_n x_n \Big|>a\Big)\leq (1-a^2)^2/3
 \ \Rightarrow \ \sum_{n\geq 1} |x_n|^2\leq 1\]
\end{proposition}

We need the Paley-Zygmund inequality (see \cite[Corollary
3.3.2]{delaPGi}) which says that for a positive nonzero random variable $\xi:\O\to \R$ and $q\in (2, \infty)$ one has
\[\P(\xi>\lambda \|\xi\|_2) \geq \Big[(1-\lambda^2) \frac{\|\xi\|_2^2}{\|\xi\|_q^2}\Big]^{q/(q-2)} \ \ \lambda\in [0,1].\]

\begin{proof}
Assume $\sum_{n\geq 1} x_n^2> 1$.
Let $\xi = \Big|\sum_{n\geq 1} r_n x_n\Big|$ and $m := \|\xi\|_2>1$. Recall the following case of Khintchine's inequality: $\E\xi^4 \leq 3
(\E \xi^2)^2$ (see \cite[Section 1.3]{delaPGi}). Therefore, the Paley-Zygmund
inequality applied shows that
\[\P(\xi>a)\geq \P(\xi>a\|\xi\|_2) \geq (1-a^2)^2 \frac{(\E\xi^2)^2}{\E\xi^4} \geq (1-a^2)^2/3.
\]
\end{proof}

We will also need the following lemma.

\begin{lemma}\label{lem:nonzero}
Let $\eta = \sum_{n\geq 1} r_n x_n$, with $\sum_{n\geq 1} x_n^2\in (0,\infty)$. Then $\P(\eta=0)\leq 1- 2 e^{-2+\gamma}\approx 0.517$, where $\gamma$ is Euler constant.
\end{lemma}
Note that for $\eta = r_1+r_2$ one has $\P(\eta = 0) = 1/2$, which shows that the lemma is close to optimal.
\begin{proof}
By scaling we can assume $\|\eta\|_2=1$. By the Paley-Zygmund inequality applied with $\xi=|\eta|$ together with the best constant in the Khintchine inequality (see \cite{Haag}) one sees that for all $\lambda\in (0,1)$ and $q>2$,
\[\P(|\eta|>\lambda) = \P(\xi>\lambda)\geq \Big[(1-\lambda^2) B_q^{-2}\Big]^{q/(q-2)},\]
where $B_q = \sqrt{2} \Big(\frac{\Gamma((p+1)/2)}{\sqrt{\pi}}\Big)^{1/q}$. An elementary calculation for $\Gamma$-functions shows that
$B_q^{-2q/(q-2)}\to 2 e^{-2+\gamma}$ as $q\downarrow 2$. Now the result follows by first taking $q>2$ arbitrary close to $2$ and then $\lambda$ small enough.
\end{proof}

\begin{proof}[Proof of Theorem \ref{thm:main}]
The second estimate follows from H\"older's inequality with $\frac1p = \frac1q + \frac1r$ and the unweighted Khintchine inequality with constant $k_{r,2}$:
\[
\|w \xi\|_p\leq \|w\|_{q} \|\xi\|_r\leq \|w\|_{q} k_{r,2} \Big(\sum_{n\geq 1} x_n^2\Big)^{\frac12}.
\]

Next we prove the first estimate. Since $\|w \xi\|_p$ increases in $p$, it
suffices to consider $p\in (0,2]$. If all the $x_n$ are zero, there is nothing to prove. If not, then by Lemma \ref{lem:nonzero} and the assumption we have $\P(w \xi\neq 0) = \P(w \neq 0, \xi\neq 0) >0$, and therefore $\|w\xi\|_p>0$. To complete the proof we can assume that $\|w\xi\|_p = 1$ as follows by a scaling argument. Moreover, by replacing $w$ by $|w|$ if necessary, we can assume that $w$ is nonnegative.

Choose $a\in (0,1)$ so small that $b = (1-a^2)^2/3>1-s$, where $s=\P(w\neq 0)$. (For example take $a$ such that $b=(1-a^2)^2/3=[(1-s) + 1/3]/2$). 
Let \[\delta_0 = \sup\{\delta>0: \P(w>\delta) \geq (s+1-b)/2\}.\]
Since $\P(w>0)=s>(s+1-b)/2$ we have $\delta_0>0$. Let $A = \{w\geq \delta_0\}$.
Then it follows that for all $t>0$
\begin{align*}
\P(\{|\xi|>t\}\cap A)& = \P(\one_A |\xi|>t) \leq t^{-p} \E(\one_A |\xi|^p)
\\ & \leq t^{-p} \delta_0^{-p} \E(w^p \one_A |\xi|^p) \leq t^{-p} \delta_0^{-p} \E(
|w \xi|^p)=t^{-p} \delta_0^{-p}.
\end{align*}
Therefore,
\[\P(\{|\xi|>t\}) \leq \P(\{|\xi|>t \cap A\} + \P(\O\setminus A) \leq t^{-p} \delta_0^{-p} + 1-(s+1-b)/2.\]
Now with $t = \delta_0^{-1}\Big(b-1+(s+1-b)/2\Big)^{-\frac1p}$ it follows that
$\P(\{|\xi|>t\} \leq b$. Let $y_n = \frac{a x_n}{t}$ and $\eta = \sum_{n\geq 1}
r_n y_n$. Then $\P(|\eta|>a)  = \P(\{|\xi|>t\}) \leq b$.
Therefore, Proposition \ref{prop:KhintchineL0} gives that $\sum_{n\geq 1} y_n^2
\leq 1$. In other words $\sum_{n\geq 1} x_n^2\leq  \frac{t^2}{a^2}$
and the result follows with $C_1 = a/t$.
\end{proof}


%

\begin{remark} \
\begin{enumerate}
\item A more sophisticated application of the Paley-Zygmund inequality in Proposition \ref{prop:KhintchineL0} shows that in the theorem it suffices to assume that $\P(w\neq 0)>1- 2 e^{-2+\gamma} \approx 0.517$. This is close to optimal as can be seen by taking $w = \one_{r_1+r_2 \neq 0}$ and $\xi = r_1+r_2$ for which the weighted inequality \eqref{eq:khintchineweigth} does not hold.

\item The integrability condition on $w$ used for the second estimate of \eqref{eq:khintchineweigth} can be improved. However, the general function space for $w$ is difficult to describe and not even rearrangement invariant (cf. \cite{Curb}).

\item With a similar technique one can obtain Theorem \ref{thm:main} for Gaussian random variables, $q$-stable random variables, etc.

\item The case where the $x_n$ take values in a normed space $X$, can also be considered. Then $\Big(\sum_{n\geq 1} x_n^2\Big)^{\frac12}$ has to be replaced by the $L^2$-norm $\|\xi\|_2$, where $\xi = \sum_{n\geq 1} r_n x_n$. Note that Lemma \ref{lem:nonzero} extends to this setting, as follows by applying Lemma \ref{lem:nonzero} with $\eta = \lb \xi, x^*\rb$ for a functional $x^*\in X^*$ for which $\lb \xi, x^*\rb$ is nonzero. Also the constants in Proposition \ref{prop:KhintchineL0} can be taken as before. This follows from the fact that also in the vector-valued setting $\|\xi\|_4 \leq 3^{1/4} \|\xi\|_2$ (see \cite{LatalaOlesunpub}).
\end{enumerate}
\end{remark}

\def\polhk#1{\setbox0=\hbox{#1}{\ooalign{\hidewidth
  \lower1.5ex\hbox{`}\hidewidth\crcr\unhbox0}}} \def\cprime{$'$}
\providecommand{\bysame}{\leavevmode\hbox to3em{\hrulefill}\thinspace}

\end{document}